\documentclass[a4paper,reqno,11pt]{amsart}

\usepackage{
amssymb,
amsmath,
amsthm,
eucal,
}

\theoremstyle{plain}

\newtheorem{theorem}{Theorem}[section]
\newtheorem{lemma}[theorem]{Lemma}
\newtheorem{proposition}[theorem]{Proposition}
\newtheorem{corollary}[theorem]{Corollary}
\theoremstyle{definition}

\newtheorem{remark}[theorem]{Remark}

\newtheorem{acknowledgements}{Acknowledgements\!\!}

\newcommand{\norm}[1]{{||#1||}}
\newcommand{\bignorm}[1]{{\left|\left|#1\right|\right|}}

\newcommand{\dderiv}[4]{{\partial_{#1}^{#3}\partial_{#2}^{#4}}}

\newcommand{\wtilde}[1]{{\widetilde{#1}}}
\def\supp{\mathop{\mathrm{supp}}\nolimits}
\def\Hess{\mathop{\mathrm{Hess}}\nolimits}

\def\SU{\mathop{\mathrm{SU}}\nolimits}
\def\Vol{\mathop{\mathrm{Vol}}\nolimits}
\def\R{{\mathbb{R}}}
\def\Z{{\mathbb{Z}}}
\def\N{{\mathbb{N}}}
\def\C{{\mathbb{C}}}

\def\<{{\langle}}
\def\>{{\rangle}}
\def\ep{{\varepsilon}}

\title
[Strichartz estimates for non elliptic Schr\"odinger equations]
{Strichartz estimates for non elliptic Schr\"odinger equations on compact manifolds}

\author
{Haruya Mizutani}
\address{Department of Mathematics, Graduate School of Science, Osaka University, Toyonaka, Osaka 560-0043, Japan.}
\email{haruya@math.sci.osaka-u.ac.jp}

\author
{Nikolay Tzvetkov}
\address{Universit\'e Cergy-Pontoise, UMR CNRS 8088, Cergy-Pontoise 95000, France.}
\email{nikolay.tzvetkov@u-cergy.fr}

\begin{document}



\begin{abstract}
In this note we consider the Schr\"odinger equation on compact manifolds equipped with possibly degenerate metrics. 
We prove Strichartz estimates with a loss of derivatives. The rate of loss of derivatives depends on the degeneracy of metrics. 
For the non-degenerate case we obtain, as an application of the main result, the same Strichartz estimates as that in the elliptic case. 
This extends Strichartz estimates for Riemannian metrics proved by Burq-G\'erard-Tzvetkov \cite{BGT} to the non-elliptic case and improves the result by Salort \cite{Salort}. 
We also investigate the optimality of the result for the case on $\mathbb{S}^3\times \mathbb{S}^3$. 
\end{abstract}

\maketitle

\footnotetext{H.M. is partially supported by JSPS Wakate (B) 25800083. N.T. is partially supported by the ERC grant Dispeq.}


\section{Introduction}
\label{Introduction}
In this note we consider the Schr\"odinger equation
$$
(i\partial_t+\rho^{-1}\partial_j a^{jk}\rho\partial_k)u=0;\quad u|_{t=0}=u_0,
$$
posed on a smooth compact manifold $M$ without boundaries (all compact manifolds considered in the note are without boundaries), where $\rho$ is a smooth positive function and $a^{jk}$ is a possibly degenerate uniformly bounded (co)metric on $M$, see Section \ref{section_2} for the precise definition. We then prove Strichartz estimates with a loss of derivatives, and the rate of loss of derivatives depends on the degeneracy of $a^{jk}$. 
As an application we extend Strichartz estimates for Riemannian metrics obtained by \cite{BGT} to pseudo Riemannian metrics. 

After the pioneer work of Strichartz \cite{Strichartz}, Strichartz estimates for Schr\"odinger equations have been extensively studied by many authors under various conditions on underlying manifolds. In particular, it was shown by Bourgain \cite{Bourgain} that, on the $2$-dimensional torus $\mathbb T^2=(\R/2\pi\Z)^2$, the following $L^4$-Strichartz estimate with an arbitrarily small loss holds:
\begin{align}
\label{Bourgain's_estimate}
\norm{e^{it(\partial_x^2+\partial_y^2)}u_0}_{L^4([0,1]\times \mathbb T^2)} \le C_\ep \norm{u_0}_{H^\ep(\mathbb T^2)},\quad \ep>0.
\end{align}
On arbitrary compact $n$-dimensional  Riemannian manifolds $(M,g)$ without boundary, Burq, G\'erard and the second author of this note proved in \cite{BGT} that
\begin{align}
\label{Strichartz_estimate}
\norm{e^{it\Delta_g}u_0}_{L^p([0,1];L^q(M))}\le C\norm{u_0}_{H^{\frac1p}(M)},
\end{align}
where $(p,q)$ satisfies $p\ge2$, $\frac2p+\frac nq=\frac n2$ and $(p,q)\neq(2,\infty)$. The authors also showed that, on the $n$-dimensional sphere $\mathbb S^n$, 
$$
\norm{e^{it\Delta_{\mathbb S^n}}u_0}_{L^4([0,1]\times \mathbb S^n)}\le C\norm{u_0}_{H^{s}(M)},\quad s>s(n),
$$
with $s(2)=\frac18$, $s(n)=\frac n4-\frac12$ for $n\ge3$, and that this is sharp in the sense that similar estimates fail with $s\le s(n)$ if $n\ge3$ and with $s<s(2)$ if $n=2$. 
Strichartz estimates on non-compact manifolds have also been investigated under some conditions on the geodesic flow and asymptotic conditions on the metric; see, \emph{e.g.}, \cite{MMT,BGH} and references therein. 

However, the non-elliptic case is less understood than the elliptic case. In \cite{Salort}, Salort proved Strichartz estimates with a loss of derivatives $\frac1p+\ep$ for possibly degenerate metrics on $\R^n$, where the admissible condition for a pair $(p,q)$ is different from the usual and depends on the degeneracy of the metric, see Remark \ref{remark_1}. Recently Godet and the second author \cite{Godet_Tzvetkov} (see, also Wang \cite{Wang}) showed that the solution to the following non-elliptic Schr\"odinger equation
$$
(i\partial_t+\partial_x^2-\partial_y^2)u=0;\quad u|_{t=0}=u_0,
$$
posed on $\mathbb T^2$, satisfies 
\begin{align}
\label{Godet_Tzvetkov_1}
\norm{e^{it(\partial_x^2-\partial_y^2)}u_0}_{L^p([0,1];L^q(\mathbb T^2))} \le C\norm{u_0}_{H^\frac1p(\mathbb T^2)},\quad \frac1p+\frac1q=\frac12,\ p>2.
\end{align}
Moreover, it was also proved that \eqref{Godet_Tzvetkov_1} fails if $\norm{u_0}_{H^\frac1p(\mathbb T^2)}$ is replaced by $\norm{u_0}_{H^s(\mathbb T^2)}$ with $s<\frac1p$ in contrast to \eqref{Bourgain's_estimate}. 

The purpose of the present note is extending this result to more general compact manifolds. 
To explain our result in a simple setting we first consider a non-degenerate case. 
Let $M=M_1\times M_2$ be a product of two compact Riemannian manifolds $(M_1,g_1)$ and $(M_2,g_2)$, and $g=g_1+g_2$ a canonical Riemannian metric on $M$. 
The Laplace-Beltrami operator $\Delta_g=\Delta_{g_1}+\Delta_{g_2}$, associated to $g$, is essentially self-adjoint on $C^\infty(M)$ with respect to the measure $d\mu=\sqrt{|\det g|} dx$. 
Then we consider a non-degenerate second-order differential operator $-\Delta_{g_1}+\Delta_{g_2}$ on $C^\infty(M)$, which is symmetric with respect to $d\mu$. 
A direct computation yields that $(-\Delta_{g_1}+\Delta_{g_2})^2\le (-\Delta_{g}+1)^2$ on $C^\infty(M)$ and $-\Delta_{g_1}+\Delta_{g_2}$ commutes with $-\Delta_{g}+1$. 
Hence $-\Delta_{g_1}+\Delta_{g_2}$ is essentially self-adjoint on $C^\infty(M)$ by Nelson's commutator theorem (see \cite[Theorem X. 36]{RS}) and admits a unique self-adjoint realization, which we denote by the same symbol. By Stone's theorem, $-\Delta_{g_1}+\Delta_{g_2}$ generates a unique unitary propagator $e^{it(\Delta_{g_1}-\Delta_{g_2})}$ on $L^2(M):=L^2(M,d\mu)$ such that $u(t)=e^{it(\Delta_{g_1}-\Delta_{g_2})}u_0$ is the solution to the following Schr\"odinger equation:
$$
(i\partial_t+\Delta_{g_1}-\Delta_{g_2})u=0;\quad u|_{t=0}=u_0\in L^2(M).
$$
The following statement is a simple consequence of the main result (see Theorem \ref{main_theorem_1}).
\begin{theorem}
\label{theorem_1}
Let $M$ be as above, $n=\dim M\ge2$ and $(p,q)$ satisfy the admissible condition:
\begin{align}
\label{admissible_1}
p\ge2,\quad\frac 2p+\frac nq=\frac n2,\quad (n,p,q)\neq(2,2,\infty).
\end{align}
Then there exists $C>0$ such that for any $u_0\in H^\frac1p(M)$,
\begin{align}
\label{theorem_1_1}
\norm{e^{it(\Delta_{g_1}-\Delta_{g_2})}u_0}_{L^p([0,1];L^q(M))}\le C\norm{u_0}_{H^{\frac1p}(M)}.
\end{align}
\end{theorem}
This note is organized as follows. In the next section we state the main result. Section \ref{section_3} is devoted to the proof of main theorems. In Section \ref{section_optimality}, we show that Theorem \ref{theorem_1} is sharp for the case on $\mathbb{S}^3\times\mathbb{S}^3$ in the sense that \eqref{theorem_1_1} fails if we replace $\norm{u_0}_{H^{\frac1p}(M)}$ by $\norm{u_0}_{H^{s}(M)}$ with $s<\frac1p$. 

\begin{acknowledgements}
The main part of this work has been done while the first author was visiting D\'epartement de Math\'ematiques, Universit\'e de Cergy-Pontoise. He acknowledges the hospitality of the department. The authors thank the anonymous referees for useful comments and remarks. 
\end{acknowledgements}


\section{Main result}
\label{section_2}
Let $M$ be a $n$-dimensional compact manifold with a smooth positive density $d \mu$ and set $L^p(M)=L^p(M,d\mu)$. Consider the Schr\"odinger equation
\begin{align}
\label{equation_1}
(i\partial_t-P)u=0;\quad u|_{t=0}=u_0,
\end{align}
where $P$ is a second-order differential operator  on $M$ which, in local coordinates, is of the form
\begin{align}
\label{P}
P=-\rho^{-1}\partial_{j}a^{jk}\rho\partial_k,\quad j,k=1,...,n,
\end{align}
with the summation convention, where $\rho$ is a smooth real-valued function on $M$ and $a^{jk}$ is a smooth real-valued $(2,0)$-tensor on $M$. Then we suppose the following:
\begin{itemize}
\item[(H1)] 
$\rho(x)>0$ and  
$
|\partial_x^\alpha a^{jk}(x)|+|\partial_x^\alpha \rho(x)|\le C_\alpha
$ for any $\alpha\in \Z^n_+:=(\N\cup\{0\})^n$.
\item[(H2)] In a neighborhood of each point there exists a rank $m$ submatrix $(b^{jk}(x))_{j,k=1}^m$ of $(a^{jk}(x))_{j,k=1}^n$ with $1\le m\le n$ such that
$$|\det (b^{jk}(x))_{j,k=1}^m|\ge c$$
with some positive constant $c>0$. 
\item[(H3)] $P$ is essentially self-adjoint on $C^\infty(M)$ with respect to $d \mu$. 
\end{itemize}
The hypothesis (H3) might be replaced by other assumptions assuring the existence of the dynamics.
We denote the self-adjoint realization of $P$ by the same symbol $P$. Note that these assumptions are independent of the choice of local coordinates. Indeed, if $P$ becomes $-\wtilde\rho^{-1}\partial_{j}\wtilde a^{jk}\wtilde \rho\partial_k$ after making the change of coordinates, then $\wtilde\rho$ and $\wtilde a^{jk}$ again fulfill (H1), (H2) and (H3) with the same ${m}$. For $\sigma>0$ we say that $(p,q)$ is $\sigma$-admissible if 
\begin{align}
\label{sigma_admissible}
p\ge2,\quad \frac 1p+\frac{\sigma}{q}=\frac\sigma2,\quad (\sigma,p,q)\neq(1,2,\infty). 
\end{align}

Now we are in a position to state the main result.

\begin{theorem}						
\label{main_theorem_1}
Let $3\le m\le n$. Suppose that \emph{(H1)}, \emph{(H2)} and \emph{(H3)} are satisfied and that $(p,q)$ is $\frac m2$-admissible. 
Then there exists $C>0$ such that, for any $u_0\in H^{\frac1p\left(\frac{2n}{m}-1\right)}(M)$, the solution $u(t)=e^{-itP}u_0$ to \eqref{equation_1} satisfies
\begin{align}
\label{main_theorem_1_1}
\norm{e^{-itP}u_0}_{L^p([0,1];L^q(M))}\le C\norm{u_0}_{H^{\frac1p\left(\frac{2n}{m}-1\right)}(M)}.
\end{align}
Furthermore, if there exists a smooth Riemannian metric $g_0$ on $M$ such that $P$ commutes with  $\Delta_{g_0}$, then \eqref{main_theorem_1_1} also holds for $m=n=2$ and for $n\ge3$, $m=2$. 
\end{theorem}

\begin{remark}
In the case $m=n$ the result of Theorem~\ref{main_theorem_1} extends \cite{BGT} (see also \cite{Staffilani_Tataru}) to the non-degenerate case.
\end{remark}
\begin{remark}
Let $(p,q)$ be $\frac m2$-admissible. Lemma \ref{proposition_energy_estimates} and the Sobolev inequality imply
$$
\norm{e^{-itP}u_0}_{L^q}\le C_t\norm{u_0}_{H^{\frac{2n}{pm}}(M)},\quad t\in\R.
$$
Compared to this bound we have a gain of $\frac 1p$ degrees of regularity in \eqref{main_theorem_1_1}, which is the same as for the elliptic case, even for the degenerate case. 
\end{remark}

As a corollary we also have Strichartz estimates for usual admissible pairs with a loss of derivatives depending on the degeneracy of metrics. 
\begin{corollary}
\label{main_corollary_1}
Let $n\ge2$, $\sigma\ge m/2$ and $\sigma>1$. Then, for any $\sigma$-admissible pair $(p,q)$, 
\begin{align}
\label{main_corollary_1_1}
\norm{e^{-itP}u_0}_{L^p([0,1];L^q(M))}\le C\norm{u_0}_{H^{\frac{\gamma(\sigma)}{p}}(M)},\quad u_0\in H^{\frac{\gamma(\sigma)}{p}}(M), 
\end{align}
where $\gamma(\sigma)=\frac{1}{\sigma}\left(n-\frac{m}{2}\right)$.
Furthermore, if there exists a smooth Riemannian metric $g_0$ on $M$ such that $P$ commutes with  $\Delta_{g_0}$, then \eqref{main_corollary_1_1} also holds for $(m,\sigma)=(2,1)$. In particular, for any $\frac n2$-admissible pair $(p,q)$ we have
$$
\norm{e^{-itP}u_0}_{L^p([0,1];L^q(M))}\le C\norm{u_0}_{H^{\frac 1p\left(2-\frac mn\right)}(M)},\quad u_0\in H^{\frac 1p\left(2-\frac mn\right)}(M).
$$
\end{corollary}

\begin{proof}
We fix $p\ge2$ arbitrarily and let $q_m=\frac12-\frac{2}{mp}$. An interpolation between \eqref{main_theorem_1_1} with the trivial bound
$
\norm{e^{-itP}u_0}_{L^p([0,1];L^2(M))}=\norm{u_0}_{L^2(M)}
$
implies
$$
\norm{e^{-itP}u_0}_{L^p([0,1];L^{q_\theta}(M))}\le C\norm{u_0}_{H^{\frac {s_\theta}{p}}(M)},
$$
where $\theta\in[0,1]$ and 
$$
\frac{1}{q_\theta}=\frac{\theta}{q_m}+\frac{1-\theta}{2}=\frac12-\frac{2\theta}{mp},\quad
s_\theta=\left(\frac{2n}{m}-1\right)\theta.
$$
Choosing $\sigma=\frac{m}{2\theta}$ we obtain desired estimates.
\end{proof}

\begin{remark}
\label{remark_1}
While this note is concerned with the compact manifold case only, Theorem \ref{main_theorem_1} still holds for the operator $P$ of the form \eqref{P} posed on $\R^n$ such that $P$ satisfies (H1), (H2) and (H3). The proof is essentially same. Since $\gamma(n-\frac m2)=1$ our result is an improvement of \cite{Salort} in which Strichartz estimates with a loss of derivatives $\frac1p+\ep$ for arbitrarily small $\ep>0$ have been proved for $(n-\frac m2)$-admissible pair, under the additional assumption that the energy estimate is satisfied (see Lemma \ref{proposition_energy_estimates}).  
\end{remark}

\section{Proof of Theorem \ref{main_theorem_1}}
\label{section_3}
In this section we prove Theorem \ref{main_theorem_1}. Let $g_0$ be a Riemannian metric on $M$ and $\Delta_{g_0}$ the associated Laplace-Beltrami operator. 
Note that the ellipticity of $\Delta_{g_0}$ implies the norm equivalence
$$
\norm{f}_{H^s(M)}\approx \norm{(1-\Delta_{g_0})^{s/2}f}_{L^2(M)},\quad s\in\R.
$$

Firstly we record several known results on the semiclassical functional calculus and the square function estimates with respect to the semiclassical spectral multiplier $\varphi(-h^2\Delta_{g_0})$. 
Let us recall the definition of semiclassical pseudodifferential operator ($h$-PDO for short),  which will be used throughout the paper. To a symbol $a\in C^\infty(\R^{2n})$ and $h\in(0,1]$ we associate the $h$-PDO, $a(x,hD):C_0^\infty(\R^n)\to C^\infty(\R^n)$, defined by
$$
a(x,hD)f(x)=(2\pi h)^{-n}\int e^{i(x-y)\cdot\xi/h}a(x,\xi)f(y)dyd\xi. 
$$
It is well known as the Calder\'on-Vaillancourt theorem that $a(x,hD)$ is bounded on $L^2(\R^n)$ uniformly in $h\in(0,1]$ if $a\in C_b^\infty(\R^{2n})$, \emph{i.e.}, $|\dderiv{x}{\xi}{\alpha}{\beta} a(x,\xi)|\le C_{\alpha\beta}$ for any $\alpha,\beta$. Moreover, if $a\in C_0^\infty(\R^{2n})$ then, for any $1\le p\le q\le\infty$, $a(x,hD)$ is bounded from $L^p(\R^n)$ to $L^q(\R^n)$:
$$
\norm{a(x,hD)f}_{L^q(\R^n)}\le C_{pq}h^{-n\left(\frac1p-\frac1q\right)}\norm{f}_{L^p(\R^n)},\quad h\in(0,1]. 
$$

Using $h$-PDO's, one obtains an approximation theorem of the semiclassical spectral multiplier and the square function estimates. 

\begin{proposition}[\cite{BGT}]
\label{proposition_functional_calculus}
Let $\varphi\in C_0^\infty(\R)$, $\kappa:\R^n\supset U\to V\subset M$ a  coordinate patch  and $\chi_1,\chi_2\in C_0^\infty(V)$ such that $\chi_2\equiv1$ near $\supp \chi_1$. Then  there exists a sequence $\psi_j\in C_0^\infty(U\times\R^n;\R)$ such that the following statements are satisfied:
\begin{itemize}
\item $\psi_j$ are supported in $\{(x,\xi)\in U\times\R^n;\ \kappa(x)\in \supp \chi_1,\ |\xi|_{g_0}^2\in \supp\varphi\}$, where 
$$
|\xi|_{g_0}^2=g_0^{jk}(x)\xi_j\xi_k,\quad (g_0^{jk})=(g_{0,jk})^{-1},\ (x,\xi)\in T^*M,
$$ is the principal symbol of $-\Delta_{g_0}$.
\item  For any integer $N\ge0$ and $\sigma\in[0,N]$ there exists $C_{N,\sigma}>0$ such that for any $h\in(0,1]$, 
$$
\bignorm{\chi_1\varphi(-h^2\Delta_{g_0})f-\sum_{j=0}^{N-1}h^j\kappa_*\psi_j(x,hD)\kappa^*(\chi_2f)}_{H^\sigma(M)}\le C_{N,\sigma}h^{N-\sigma}\norm{f}_{L^2(M)}.
$$
In particular, for any $1\le p\le q\le\infty$ there exists $C_{p,q}>0$ such that for any $h\in(0,1]$, 
$$
\norm{\varphi(-h^2\Delta_{g_0})f}_{L^q(M)}\le C_{p,q}h^{-n\left(\frac1p-\frac1q\right)}\norm{f}_{L^p(M)}.
$$
\end{itemize}
\end{proposition}

\begin{proposition}[\cite{BGT}]
\label{proposition_LP}
Consider a $4$-adic partition of unity on $[1,\infty)$:
$$
\varphi\in C_0^\infty(\R),\quad\supp\varphi\subset(1/4,4),\quad0\le \varphi\le 1,\quad \sum_{j=0}^\infty\varphi(2^{-2j}\lambda)=1\ \text{for}\ \lambda\ge1.
$$
Then, for any $q\in[2,\infty)$, 
$$
\norm{v}_{L^q(M)}\le C_q\norm{v}_{L^2(M)}+C_q\Big(\sum_{j=0}^\infty\norm{\varphi(-2^{-2j}\Delta_{g_0})v}_{L^q(M)}^2\Big)^{1/2}.
$$
\end{proposition}

For the proof of Propositions \ref{proposition_functional_calculus} and \ref{proposition_LP}, we refer to \cite{BGT}. If $P$ commutes with $\Delta_{g_0}$ (and so hence does with $\varphi(-h^2\Delta_{g_0})$), then using Proposition \ref{proposition_LP}, the energy conservation 
$\varphi(-h^2\Delta_{g_0})e^{-itP}=e^{-itP}\varphi(-h^2\Delta_{g_0})$ 
and the norm equivalence 
\begin{align}
\label{almost_orthogonality}
\norm{f}_{H^s(M)}\approx \norm{f}_{L^2(M)}+\Big(\sum_{j=0}^\infty 2^{-2sj}\norm{\varphi(-2^{-2j}\Delta_{g_0})f}_{L^2(M)}^2\Big)^{\frac12},
\end{align}
one can see by a similar argument as that in \cite{BGT} that Theorem \ref{main_theorem_1} follows from the following semiclassical Strichartz estimates:
$$
\norm{\varphi(-h^2\Delta_{g_0})e^{-itP}u_0}_{L^p([0,1];L^q(M))}\le Ch^{-\frac1p\left(\frac{2n}{m}-1\right)}\norm{u_0}_{L^2(M)},\quad h\in(0,1].
$$
However, this is not in the case for our Hamiltonian $P$ in general. To overcome this difficulty we use the following energy estimate instead of the energy conservation. 

\begin{lemma}
\label{proposition_energy_estimates}
For any $s\in\R$ there exists $C_s>0$ such that for any $u_0\in H^s(M)$ and $t\in\R$,
$$
\norm{e^{-itP}u_0}_{H^s(M)}\le e^{C_s|t|}\norm{u_0}_{H^s(M)}.
$$
\end{lemma}

\begin{proof}
We may assume $t\ge0$ without loss of generality since the proof for the opposite case is analogous. We put $v=(1-\Delta_{g_0})^{\frac s2}e^{-itP}u_0$ which solves
$$
(i\partial_t-P)v=Bv;\quad v|_{t=0}=(1-\Delta_{g_0})^{\frac s2}u_0,
$$
where $B=[(1-\Delta_{g_0})^{\frac s2},P](1-\Delta_{g_0})^{-\frac s2}$ and its symbol, in local coordinates, is given by
$$
\{(1+|\xi|_{g_0}^2)^{\frac{s}{2}},|\xi|_a^2\}(1+|\xi|_{g_0}^2)^{-\frac{s}{2}}+r_1(x,\xi)
$$ 
with some symbol $r_1$ of order zero, where $|\xi|_a^2:=a^{jk}(x)\xi_j\xi_k$ is the principal symbol of $P$ and $\{\cdot,\cdot\}$ is the Poisson bracket (see, \emph{e.g.}, H\"ormander \cite{Hormander}). We also learn by the symbolic calculus that the symbol of the adjoint $B^*$ is of the form 
$$
\{(1+|\xi|_{g_0}^2)^{\frac{s}{2}},|\xi|_a^2\}(1+|\xi|_{g_0}^2)^{-\frac{s}{2}}+r_2(x,\xi)
$$
with some symbol $r_2$ of order zero. In particular, $B-B^*$ has a bounded symbol and hence is bounded on $L^2(M)$ by the Calder\'on-Vaillancourt theorem. Now we compute
\begin{align*}
\frac{d}{dt}\norm{v(t)}_{L^2(M)}^2
&=\<-i(P+B)v(t),v(t)\>+\<v(t),-i(P+B)v(t)\>\\
&=-i\<(B-B^*)v(t),v(t)\>\\
&\le C_s\norm{v(t)}_{L^2(M)}^2.
\end{align*}
Integrating with respect to $t$ then implies the assertion. 
\end{proof}

We next state the following micro-localized dispersion estimate which is crucial to prove semiclassical Strichartz estimates.

\begin{proposition}
\label{proposition_dispersive}
For any $\varphi\in C_0^\infty(\R)$ supported in $(0,\infty)$ and for sufficiently small $\alpha>0$, there exists $C_\alpha>0$ such that for any $h\in(0,1]$, $|t|\le\alpha h$ and any $u_0\in C_0^\infty(M)$, 
\begin{align}
\label{proposition_dispersive_1}
\norm{e^{-itP}\varphi(-h^2\Delta_{g_0})u_0}_{L^\infty(M)}\le C_\alpha |t|^{-\frac {m}2}h^{{m}-n}\norm{u_0}_{L^1(M)}.
\end{align}
\end{proposition}

\begin{proof}
The proof is essentially same as that in Salort \cite{Salort} (see also \cite[Lemma 2.7]{BGT}) and we hence outline it only. 
Let $\kappa: \R^n\Supset U\to V\Subset M$ be a coordinate patch and let $U_1\Subset U$. Choosing a sufficiently small coordinate neighborhood $U$ if necessary we may assume without loss of generality that $P$ may be of the form \eqref{P} and satisfies (H1) and (H2) in $U$. By virtue of Proposition \ref{proposition_functional_calculus} and the partition of unity argument, it suffices to show that the following dispersion estimate holds:
\begin{align}
\label{proof_dispersive_1}
\norm{e^{-ithP}\psi(x,hD)f}_{L^\infty(\R^n)}\le C_\alpha |th|^{-\frac {m}2}h^{{m}-n}\norm{f}_{L^1(\R^n)},
\end{align}
where $|t|\le\alpha$, $h\in(0,1]$, $f\in C_0^\infty(U)$ and $ \psi\in C_0^\infty(U_1\times\R^n)$. To prove this bound we use the semiclassical WKB approximation. It has been proved by Salort \cite[Propositions 3 and 4, and Lemma 3]{Salort} that for any integer $N\ge1$ there exists $\alpha_N>0$ such that
$$
\sup_{|t|<\alpha_N}\norm{\left(e^{-ithP}\psi(x,hD)-J_N(t,h)\right)f}_{L^\infty(\R^n)}\le C_Nh^{N-\frac{5n}{2}}\norm{f}_{L^1(\R^n)},
$$
where $J_N(t,h)$ is a semiclassical Fourier integral operator of the form
$$
J_N(t,h)f(x)=\frac{1}{(2\pi h)^{n}}\int e^{\frac ih(\Phi(t,x,\xi)-y\cdot\xi)}\sum_{j=0}^{N-1}h^ja_j(t,x,\xi) f(y)dyd\xi. 
$$
with $a_j\in C_0^\infty((-\alpha_N,\alpha_N)\times U_1\times\R^{n})$ and $\Phi\in C^\infty((-\alpha_N,\alpha_N)\times U_1\times\R^{n})$ satisfying
$$
\frac{\partial_{\xi_j}\partial_{\xi_k}\Phi(t,x,\xi)}{t}=-a^{jk}(x)+O(|t|)\ \text{as}\ |t|\to0. 
$$
By (H2), the partial Hessian of $\Phi/t$ in $m$ directions, $\Hess_m(\Phi/t)=-(b^{jk}(x))_{j,k=1}^m+O(|t|)$, is non-degenerate if $\alpha_N$ is small enough. Therefore we can apply the stationary phase method to obtain
$$
\norm{J_N(t,h)f}_{L^\infty(\R^n)}\le C_{N,m}|th^{-1}|^{-\frac m2}h^{-n}\norm{f}_{L^1(\R^n)},\quad |t|\le\alpha_N,\ h\in(0,1],
$$
which, together with the above error bound, implies \eqref{proof_dispersive_1} provided that $N>\frac{5n}{2}$. 
\end{proof}

Set 
$
\gamma=\frac{2n}{m}-1. 
$ 
Since $\varphi(-h^2\Delta_{g_0})$ is bounded on $L^\infty(M)$, using the change of variable $t\mapsto (t-s)h^{1-\gamma}$ and \eqref{proposition_dispersive_1}, we have
$$
\norm{\varphi(-h^2\Delta_{g_0})e^{-ih^{1-\gamma}(t-s)P}\varphi(-h^2\Delta_{g_0})u_0}_{L^\infty(M)}\le C_\alpha |t-s|^{-\frac m2}\norm{u_0}_{L^1(M)}
$$
for $|t-s|\le\alpha h^{\gamma}$ and $t\neq s$. The $TT^*$-argument due to \cite{Keel_Tao} then provides the following.

\begin{corollary}
\label{corollary_proof_main_theorem_2}
Let $n\ge2$, $\varphi$ as above. Then, for any interval $I_{h}$ with length $|I_{h}|\lesssim {h}$ and any $\frac m2$-admissible pair $(p,q)$ there exists $C>0$ such that for any $h\in(0,1]$, 
\begin{align*}
\norm{\varphi(-h^2\Delta_{g_0})e^{-itP}f}_{L^p(I_h;L^{q}(M))}
&\le Ch^{\frac {1-\gamma}{ p}}\norm{f}_{L^2(M)},\\
\bignorm{\int_{s<t}\chi_{I_h}(s)\varphi(-h^2\Delta_{g_0})e^{-i(t-s) P}\varphi(-h^2\Delta_{g_0})F(s)ds}_{L^p(I_h;L^{q}(M))}
&\le Ch^{\frac{1-\gamma}{ p}}\norm{F}_{L^1(I_h;L^2(M))},
\end{align*}
where $\chi_{I_h}$ is the characteristic function of $I_h$. 
\end{corollary}

\begin{proof}[Proof of Theorem \ref{main_theorem_1}]
If $P$ commutes with $\Delta_{g_0}$ then \eqref{main_theorem_1_1} easily follows from Corollary \ref{corollary_proof_main_theorem_2} and Proposition \ref{proposition_LP}. We hence consider the case when $P$ does not commutes with $\Delta_{g_0}$ only. 

The proof is a slight modification of \cite[Proposition 5.4]{Bouclet_Tzvetkov_1}. Let $3\le m\le n$. By virtue of the interpolation theorem, it suffices to show the endpoint estimate:
$$
\norm{e^{-itP}u_0}_{L^2([0,1];L^{2_m^*}(M))}
\le C\norm{u_0}_{H^{\frac{\gamma}{2}}(M)},\quad
2_m^*:=\frac{2m}{m-2}.
$$ 
Let $\varphi,\psi\in C_0^\infty(\R)$ be such that $\supp \varphi,\supp\psi \Subset(0,\infty)$ and $\psi\equiv1$ on $\supp \varphi$. We write $\varphi_h:=\varphi(-h^2\Delta_{g_0})$ and $\psi_h:=\psi(-h^2\Delta_{g_0})$ for simplicity. 
Since $u=\varphi_he^{-it P}u_0$ solves
$$
(i\partial_t-P)u=[\varphi_h,P]e^{-itP}u_0;\quad u|_{t=0}=\varphi_hu_0,
$$
we obtain the Duhamel formula
$$
\varphi_he^{-itP}
=e^{-itP}\varphi_h
-i\int_0^t e^{-i(t-s)P}[\varphi_h,P]e^{-is P}ds.
$$
By virtue of Proposition \ref{proposition_functional_calculus}, the symbol of $[\varphi_h,P]$ is supported in $\supp \varphi(h^2|\xi|_{g_0}^2)$ modulo $O(h^\infty)$ and $\psi(h^2|\xi|_{g_0}^2)\equiv1$ in $\supp \varphi(h^2|\xi|_{g_0}^2)$. We thus learn by the symbolic calculus that

$$
[\varphi_h,P]=\psi_h[\varphi_h,P]\psi_h+O_{L^2\to L^{2_m^*}}(h),
$$
which implies
\begin{align}
\nonumber
\varphi_he^{-itP}
&=\psi_h\varphi_he^{-itP}\\
\label{proof_theorem_1}
&=\psi_he^{-itP}\varphi_h
-i\int_0^t \psi_he^{-i(t-s)P}\psi_h[\varphi_h, P]\psi_he^{-is P}ds
+O_{L^2\to L^{2_m^*}}(h).
\end{align}

Next we split the time interval $[0,1]=J_0\cup J_1\cup\cdots\cup J_N$ with $J_j=[jh,(j+1)h]$, $j=0,1,...,N-1$, and $J_N=[1-h,1]$. For $j=0$, applying Corollary \ref{corollary_proof_main_theorem_2} to \eqref{proof_theorem_1} yields
\begin{align*}
&\norm{\varphi_h e^{-itP}u_0}_{L^2(J_0;L^{2_m^*}(M))}^2\\
&\le Ch^{1-\gamma}\norm{\varphi_hu_0}_{L^2(M)}^2
+Ch\norm{u_0}_{L^2(M)}^2
+Ch^{1-\gamma}\norm{[\varphi_h,P]\psi_he^{-itP}u_0}_{L^1(J_0;L^{2}(M))}^2\\
&\le Ch^{1-\gamma}\norm{\varphi_hu_0}_{L^2(M)}^2
+Ch\norm{u_0}_{L^2(M)}^2
+Ch^{-\gamma}\norm{\psi_he^{-itP}u_0}_{L^2(J_0;L^{2}(M))}^2,
\end{align*}
where, in the last line, we have used the bound $[\varphi_h, P]=O_{L^2\to L^2}(h^{-1})$, H\"older's inequality with respect to $t$ and the bound $|J_0|\le h$. 
We similarly obtain the same bound for $j=N$:
\begin{align*}
&\norm{e^{-itP}u_0}_{L^2(J_N;L^{2_m^*}(M))}^2\\
&\le Ch^{1-\gamma}\norm{\varphi_hu_0}_{L^2(M)}^2
+Ch\norm{u_0}_{L^2(M)}^2
+Ch^{-\gamma}\norm{\psi_he^{-itP}u_0}_{L^2(J_N;L^{2}(M))}^2.
\end{align*}
For $j=1,2,...,N-1$, taking $\theta \in C_0^\infty(\R)$ so that $\theta\equiv1$ on $[-1/2,1/2]$ and $\supp \theta \subset[-1,1]$, we set $\theta_j(t)=\theta(t/h-j-1/2))$. 
It is easy to see that $\theta_j\equiv1$ on $J_j$ and $\supp \theta_j\subset \wtilde{J}_j=J_j+[-h/2,h/2]$. 
Then we consider $v_j=\theta_j(t)\varphi_he^{-itP}u_0$, which solves
$$
(i\partial_t -P)v_j=\theta_j' u+\theta_j[\varphi_h,P]e^{-itP}u_0;\quad v_j|_{t=0}=0,
$$
and hence obeys the Duhamel formula
$$
\theta_j(t)\varphi_he^{-itP}=-i\int_0^t e^{-i(t-s)P}\left(\theta_j'(s)\varphi_h+\theta_j(s)[\varphi_h,P]e^{-is P}\right)ds.
$$
A same argument as in the case $j=0$ and Corollary \ref{corollary_proof_main_theorem_2} then imply 
\begin{align*}
&\norm{\varphi_he^{-itP}u_0}_{L^2(J_j;L^{2_m^*}(M))}^2\\
&\le \norm{\theta_j(t)\varphi_he^{-itP}u_0}_{L^2(\wtilde J_j;L^{2_m^*}(M))}^2\\
&\le Ch\norm{u_0}_{L^2(M)}^2
+ Ch^{-1-\gamma}\norm{\varphi_he^{-itP}u_0}_{L^1(\wtilde{J}_j;L^2(M))}^2
+ Ch^{1-\gamma}\norm{[\varphi_h,P]\psi_he^{-itP}u_0}_{L^1(\wtilde{J}_j;L^2(M))}^2\\
&\le Ch\norm{u_0}_{L^2(M)}^2
+Ch^{-\gamma}\norm{\psi_he^{-itP}u_0}_{L^2(\wtilde{J}_j;L^{2}(M))}^2.
\end{align*}
Summing over $j=0,1,...,N$ implies
$$
\norm{\varphi_he^{-itP}u_0}_{L^2([0,1];L^{2_m^*}(M))}^2
\le Ch^{1-\gamma}\norm{\varphi_hu_0}_{L^2(M)}^2
+Ch^{-\gamma}\norm{\psi_he^{-itP}u_0}_{L^2([0,1];L^2(M))}^2.
$$
Summing over $h=2^{-j}$ with $j\in \N$, using Proposition \ref{proposition_LP} and \eqref{almost_orthogonality} we obtain
$$
\norm{e^{-itP}u_0}_{L^2([0,1];L^{2_m^*}(M))}
\le C\norm{u_0}_{H^{\frac{-1+\gamma}2}(M)}+C\norm{e^{-itP}u_0}_{L^2([0,1];H^{\frac{\gamma}{2}}(M))}.
$$ 
Applying Lemma \ref{proposition_energy_estimates} to the second term of the right hand side implies 
\begin{align*}
\norm{e^{-itP}u_0}_{L^2([0,1];H^{\frac{\gamma}{2}}(M))}
\le C\norm{u_0}_{H^{\frac{\gamma}{2}}(M)},
\end{align*}
which completes the proof.
\end{proof}

\section{Optimality of Theorem \ref{theorem_1}}
\label{section_optimality}
In this section we prove the following, which shows that the loss of derivatives $\frac1p$ in Theorem \ref{theorem_1} is sharp for the case on $\mathbb{S}^3\times\mathbb{S}^3$. 

\begin{proposition}
\label{proposition_optimality_1}
Suppose that $\mathbb{S}^3$ is endowed with the standard metric. Let $(p,q)$ satisfy the admissible condition \eqref{admissible_1} with $n=6$. Then the estimate
\begin{align}
\label{proposition_optimality_1_1}
\norm{e^{it(\Delta_x-\Delta_y)}u_0}_{L^p([0,1];L^q(\mathbb{S}^3_x\times \mathbb{S}^3_y))}
\le C\norm{u_0}_{H^s(\mathbb{S}^3_x\times \mathbb{S}^3_y)}
\end{align}
fails for $s<\frac1p$, where $\Delta_x=\Delta_{\mathbb{S}^3_x}$ and $\Delta_y=\Delta_{\mathbb{S}^3_y}$. 
\end{proposition}

The proof  basically follows the same line to \cite[Lemma 2.1]{Godet_Tzvetkov}. In order to prove this proposition, isometries on $\mathbb{S}^3$ play a crucial role. We first recall the invariance of the Laplacian with respect to isometries:
\begin{lemma}
\label{lemma_optimality_2}
Let $(M,g)$ be a $d$-dimensional smooth compact Riemannian manifold without boundaries and $\Delta_g$ the associated Laplace-Beltrami operator. Let $\varphi$ be an isometry of $M$. Then, for any smooth function $u$ on $M$ and every $x\in M$, $\Delta_g(u(\varphi(x)))=\varphi(\Delta_g u(x))$. 
\end{lemma}

\begin{proof}
It suffices to show that, for any smooth functions $u$ and $v$ on $M$,
\begin{align}
\label{ee}
\int_M \Delta_g(u(\varphi(x)))v(\varphi(x))dg(x)=\int_M (\Delta_g u)(\varphi(x))v(\varphi(x))dg(x),
\end{align}
where $dg(x)={\det(g_{ij}(x))}^{\frac{1}{2}} dx={|g(x)|}^{\frac{1}{2}} dx$. Since $\varphi$ is an isometry and hence preserves the volume element, the right hand side is equal to
$\int_M \Delta u(x)v(x)dg(x)$. In local coordinates, this may be written in the form
\begin{eqnarray*}
\int_M {\frac{1}{{|g(x)|^{\frac{1}{2}}}}}\partial_i \left(g^{ij}(x) \partial_ju(x) {|g(x)|}^{\frac{1}{2}}\right) v(x){|g(x)|}^{\frac{1}{2}} dx
= -\int_M  g^{ij}(x) (\partial_ju)(x) {|g(x)|}^{\frac{1}{2}}(\partial_i v)(x)dx.
\end{eqnarray*}
On the other hand, the left hand side of (\ref{ee}) can be brought to the form
\begin{align*}
&\int_M  \partial_i \left(g^{ij}(x) \partial_ju(\varphi(x)) {|g(x)|}^{\frac{1}{2}}\right) v(\varphi(x)) dx
= -\int_M g^{ij}(x) \partial_j u(\varphi(x)) {|g(x)|}^{\frac{1}{2}} \partial_i v(\varphi(x)) dx \\
&= -\int_M g^{ij}(x) \Big((\partial_k u)(\varphi(x)) \partial_j \varphi_k (x)\Big)
\Big( (\partial_l v)(\varphi(x)) \partial_i \varphi_l (x)\Big) {|g(x)|}^{\frac{1}{2}}dx,
\end{align*}
where $\varphi_1,\ldots, \varphi_d$ are the coordinates in $\mathbb{R}^d$ of the function $\varphi$. 
Now let $$\tilde{g}^{kl}(x)= g^{ij}(x) \partial_j \varphi_k (x)\partial_i \varphi_l (x).$$ Then $\tilde{g}$ is the metric induced by the map $\varphi$, and since $\varphi$ is an isometry, $\tilde{g}(x)=g(\varphi(x))$. Hence the left hand side of (\ref{ee}) is equal to 
\begin{eqnarray*}
\int_M  \tilde{g}^{kl}(x) (\partial_k u)(\varphi(x)) (\partial_l v)(\varphi(x)){|g(x)|}^{\frac{1}{2}}dx
&=& -\int_M  {g}^{kl}(\varphi(x)) (\partial_k u)(\varphi(x)) (\partial_l v)(\varphi(x)){|g(x)|}^{\frac{1}{2}}dx\\
&=& -\int_M  {g}^{kl}(x) (\partial_k u)(x) (\partial_l v)(x){|g(x)|}^{\frac{1}{2}}dx,
\end{eqnarray*}
again because $\varphi$ is an isometry. This proves (\ref{ee}).
\end{proof}

\begin{proof}[Proof of Proposition \ref{proposition_optimality_1}]
Let us recall the group structure of $\mathbb{S}^3$. We can view $\mathbb{S}^3$ as the unit sphere in the quaternion field and this endows $\mathbb{S}^3$ with a group structure with the identity element $(1,0,0,0)$. More precisely, using a bijective homomorphism 
\begin{align*}
&\mathbb{S}^3\ni x=(x_1,x_2,x_3,x_4)\\
&\mapsto \left(\begin{matrix}x_1+ix_2  & x_3+ix_4\\-x_3+ix_4 & x_1-ix_2\end{matrix}\right)\in \SU(2)=\left\{\left(\begin{matrix}\alpha  & \beta\\-\overline\beta & \overline\alpha\end{matrix}\right)\in \C^{2\times 2};\ |\alpha|^2+|\beta|^2=1\right\},
\end{align*}
we have $\mathbb{S}^3\cong \SU(2)$ and the induced group law on $\mathbb{S}^3$ is given by
\begin{align*}
(x_1,x_2,x_3,x_4)\cdot(y_1,y_2,y_3,y_4)
&=(
x_1y_1-x_2y_2-x_3y_4-x_4y_4,
x_1y_2+x_2y_1+x_3y_4-x_3y_4,\\
&x_1y_3+x_3y_1+x_3y_4-x_4y_3,
x_1y_4+x_4y_1+x_2y_3-x_3y_2),\\
(x_1,x_2,x_3,x_4)^{-1}&=(x_1,-x_2-x_3-x_4).
\end{align*}
Note that the right and left multiplication maps $R_x,L_x:\mathbb{S}^3\to\mathbb{S}^3$, defined by $R_xy=y\cdot x$ and $L_xy=x\cdot y$, respectively,  are orthogonal maps and hence isometries with respect to the standard metric on $\mathbb{S}^3$. Furthermore, if $f\in C^\infty(\mathbb{S}^3)$ then $f(x\cdot y)\in C^\infty(\mathbb{S}_x^3\times\mathbb{S}_y^3)$ and $f(x\cdot y)$ satisfies the following stationary problem:
\begin{align*}
(\Delta_x-\Delta_y)f(x\cdot y)=\Delta_x(f(L_yx))-\Delta_y(f(R_xy))=(\Delta f)(L_yx)-(\Delta f)(R_xy)=0,
\end{align*}
where we have used Lemma \ref{lemma_optimality_2}. Hence if \eqref{proposition_optimality_1_1} holds then
\begin{align}
\label{optimality_1}
\norm{f(x\cdot y)}_{L^q(\mathbb{S}^3_x\times\mathbb{S}^3_y)}\le C\norm{f(x\cdot y)}_{H^s(\mathbb{S}_x^3\times\mathbb{S}_y^3)}.
\end{align}
Since isometries preserve the volume element, using the change of variable $x\mapsto x\cdot y^{-1}$, we have $\norm{f(x\cdot y)}_{L^q(\mathbb{S}^3_x\times\mathbb{S}^3_y)}^q=\Vol(\mathbb{S}^3)\norm{f}_{L^q(\mathbb{S}^3)}^q$. By an interpolation and Lemma \ref{lemma_optimality_2} we also obtain 
$$
\norm{f(x\cdot y)}_{H^s(\mathbb{S}_x^3\times\mathbb{S}_y^3)}
\le C \norm{f(x\cdot y)}_{L^2(\mathbb{S}^3_x\times\mathbb{S}^3_y)}^{1-\frac s2}\norm{f(x\cdot y)}_{H^2(\mathbb{S}^3_x\times\mathbb{S}^3_y)}^{\frac s2}
\le C \norm{f}_{L^2(\mathbb{S}^3)}^{1-\frac s2}\norm{f}_{H^2(\mathbb{S}^3)}^{\frac s2}. 
$$
Therefore \eqref{optimality_1} implies
\begin{align}
\label{optimality_2}
\norm{f}_{L^q(\mathbb{S}^3)}\le C \norm{f}_{L^2(\mathbb{S}^3)}^{1-\frac s2}\norm{f}_{H^2(\mathbb{S}^3)}^{\frac s2}.
\end{align}
Next we choose $\psi \in C_0^\infty(\R^3)$ supported in a sufficiently small ball $U$ centered at the origin and, with a large parameter $\lambda\ge1$, define $\varphi_\lambda\in C^\infty(\mathbb{S}^3)$ by
$$
\varphi_\lambda(x_1,x_2,x_3,x_4)=
\begin{cases}
\psi(\lambda x_2,\lambda x_3,\lambda x_4),&(x_2,x_3,x_4)\in U_\lambda,\\
0,&\text{otherwise},
\end{cases}
$$
where $U_\lambda=\{z\in \R^3;\ \lambda z\in U\}\subset U$. Note that $\varphi_\lambda$ is supported near $(1,0,0,0)$. Then the $L^q$ norm of $\varphi_\lambda$ reads
$$
\Big(\int_{\mathbb{S}^3}|\varphi_\lambda(x)|^qd\mu(x)\Big)^{\frac 1q}\approx \Big(\int_{\R^3}|\psi(\lambda x_2,\lambda x_3,\lambda x_4)|^qdx_2dx_3dx_4\Big)^{\frac1q}\approx \lambda^{-\frac 3q},
$$
while the sobolev norm $\norm{\varphi_\lambda}_{H^2(\mathbb{S}^3)}$ behaves like
\begin{align*}
\Big(\int_{\mathbb{S}^3}|(1-\Delta_{\mathbb{S}^3})\varphi_\lambda(x)|^2d\mu(x)\Big)^{\frac 12}\approx\Big(\int_{\R^3}|(1-\Delta_{\R^3})\psi(\lambda x_2,\lambda x_3,\lambda x_4)|^2dx_2dx_3dx_4\Big)^{\frac12}\approx \lambda^{2-\frac32},
\end{align*}
where $d\mu(x)$ denotes the Haar measure on $\mathbb{S}^3$. Now we apply \eqref{optimality_2} to $f=\varphi_\lambda$ and obtain 
$$
\lambda^{-\frac3q}\lesssim \lambda^{-\frac32\left(1-\frac s2\right)+s-\frac{3s}{4}}=\lambda^{s-\frac32}.
$$
For sufficiently large $\lambda$, this implies $s\ge 3\left(\frac12-\frac1q\right)=\frac1p$ which completes the proof. 
\end{proof}

\begin{remark}
(1) Let $G$ be a $d$-dimensional compact Lie group endowed with a bi-invariant Riemannian metric $g$. Then the above argument still works if we replace $\mathbb{S}^3\times \mathbb{S}^3$ with $G\times G$ and hence Proposition \ref{proposition_optimality_1} can be extended to the solution to
$$
(i\partial_t+\Delta_{g_x}-\Delta_{g_y})u(t,x,y)=0;\quad u|_{t=0}=u_0\in L^2(G_x\times G_y).
$$

(2) It is an interesting open problem to provide a proof of the optimality of Theorem \ref{theorem_1} without using group structures. 
\end{remark}



\begin{thebibliography}{99}
\bibitem{Bouclet_Tzvetkov_1}
Bouclet,~J.-\ M.,\ Tzvetkov,~N.\ (2007).\ 
	Strichartz estimates for long range perturbations.\ 
\emph{Amer.\ J.\ Math}.\ 129:1565--1609. 

\bibitem{Bourgain}
Bourgain,~J.\ (1993).\ 
Fourier transform restriction phenomena for certain lattice subsets and applications to nonlinear evolution equations. I. Schr\"odinger equations.\ 
\emph{Geom.\ Funct.\ Anal}.\ 3:107--156.

\bibitem{BGT}
Burq,~N.,\ G\'{e}rard,~P.,\ Tzvetkov,~N.\ (2004).\ 
	Strichartz inequalities and the nonlinear Schr\"odinger equation on compact manifolds.\ 
\emph{Amer.\ J.\ Math}.\ 126:569--605.

\bibitem{BGH}
Burq,~N.,\ Guillarmou,~C.,\ Hassell,~A.\ (2010).\  
Strichartz estimates without loss on manifolds with hyperbolic trapped geodesics.\ 
\emph{Geom. Funct.\ Anal}.\ 20:627--656.

\bibitem{Godet_Tzvetkov}
Godet,~N.,\ Tzvetkov,~N.\ (2012).\  
Strichartz estimates for the periodic non elliptic Schr\"odinger equation.\ 
\emph{C.\ R.\ Acad.\ Sci.\ Paris,\ Ser.\ I}\ 350:955--958

\bibitem{Hormander}
H\"ormander,~H.\ (1985).\ 
\emph{The Analysis of Linear Partial Differential Operators III}. Springer-Verlag.

\bibitem{Keel_Tao}
Keel,~M.,\ Tao,~T.\ (1998).\ 
	Endpoint Strichartz Estimates.\ 
\emph{Amer.\ J.\ Math}.\ 120:955--980.

\bibitem{MMT}
Marzuola,~J.,\ Metcalfe,~J.,\ Tataru,~D.\ (2008).\ 
	Strichartz estimates and local smoothing estimates for asymptotically flat Schr\"odinger equations.\ 
\emph{J.\ Funct.\ Analysis}.\ 255:1497--1553.

\bibitem{RS}
Reed,~M.,\ Simon,~B.\ (1975).\ \emph{Methods of Modern Mathematical Physics Vol. II}.\ Academic Press.


\bibitem{Salort}
Salort,~D.\ (2007).\ 
The Schr\"odinger equation type with a nonelliptic operator.\ 
\emph{Comm.\ Partial Differential Equations} 32:209--228.

\bibitem{Staffilani_Tataru}
Staffilani,~G.,\ Tataru,~D.\ (2002).\ 
	Strichartz estimates for a Schr\"odinger operator with non smooth coefficients.\ 
\emph{Comm.\ PDE}.\ 27:1337--1372.

\bibitem{Strichartz}
Strichartz,~R.\ (1977).\ 
	Restrictions of Fourier transforms to quadratic surfaces and decay of solutions of wave equations.\ 
\emph{Duke\ Math.\ J}.\ 44:705--714.

\bibitem{Wang}
Wang,~Y.\ (2013).\ 
Periodic cubic hyperbolic Schr\"odinger equation on $\mathbb T^2$.\ 
\emph{J.\ Funct.\ Anal}.\ 265:424--434.
\end{thebibliography}
\end{document}